%% file: MultiChannelSeparation.tex
\documentclass[11pt]{article}

\usepackage{amsmath,amssymb,graphicx,cite,color}

\numberwithin{equation}{section}

\graphicspath{{./Figs/}}

\input{defs}

\begin{document}

%------------------------------------------------------------

\title{Sparse Channel Separation using Random Probes}
\author{Justin Romberg and Ramesh Neelamani\thanks{J.\ R.\ is with the
    School of Electrical and Computer Engineering at Georgia Tech in
    Atlanta, Georgia; email: jrom@ece.gatech.edu.  R.\ N.\ is with
    ExxonMobil Exploration Company in Houston, Texas; email: ramesh.neelamani@exxonmobil.com.  This work was supported by ONR Young Investigator Award N00014-08-1-0884.}}

% draft
%\date{DRAFT: \currenttime, \today}
% for submitted version
\date{\today}

\maketitle 

\begin{abstract}
	
This paper considers the problem of estimating the channel response
	(or Green's function) between multiple source-receiver pairs.
	Typically, the channel responses are estimated one-at-a-time:
	a single source sends out a known probe signal, the receiver
	measures the probe signal convolved with the channel response,
	and the responses are recovered using deconvolution.  In this
	paper, we show that if the channel responses are sparse and
	the probe signals are random, then we can significantly reduce
	the total amount of time required to probe the channels by
	activating all of the sources simultaneously.  With all
	sources activated simultaneously, the receiver measures a
	superposition of all the channel responses convolved with the
	respective probe signals. Separating this cumulative response
	into individual channel responses can be posed as a linear
	inverse problem.

We show that channel response separation is possible (and stable) even
	when the probing signals are relatively short in spite of the
	corresponding linear system of equations becoming severely
	underdetermined. We derive a theoretical lower bound on the
	length of the source signals that guarantees that this
	separation is possible with high probability.  The bound is
	derived by putting the problem in the context of finding a
	sparse solution to an underdetermined system of equations, and
	then using mathematical tools from the theory of compressive
	sensing. Finally, we discuss some practical applications of
	these results, which include forward modeling for seismic
	imaging, channel equalization in multiple-input
	multiple-output communication, and increasing the
	field-of-view in an imaging system by using coded apertures.

\end{abstract}

%------------------------------------------------------------
\section{Introduction}

\input{Introduction}

%------------------------------------------------------------
\section{Sparse recovery from underdetermined measurements}
\label{sec:sparse-recovery}

\input{SparseRecovery}

%------------------------------------------------------------
\section{A multichannel separation theorem}
\label{sec:separation-theorem}

\input{SeparationTheorem}

%--------------------------------------------------------------
\section{Proof of Theorem~\ref{th:Emain}}
\label{sec:proofmean}

\input{ProofMean}

%--------------------------------------------------------------
\section{Proof of Theorem~\ref{th:Pmain}}
\label{sec:prooftail}

\input{ProofTail}

%--------------------------------------------------------------
% \section{Discussion}
% \label{sec:discussion}
% 
% Topics to touch on
% \begin{itemize}
% 	\item write something about how this also gives us a CS matrix which can be parallelized very nicely
% \end{itemize}

%------------------------------------------------------------

\appendix

%------------------------------------------------------------
\section{Random Matrices}

\input{RandomMatrices}

%------------------------------------------------------------
\section{A simple inequality}

\input{SimpleInequality}

%------------------------------------------------------------

\small
\bibliographystyle{plain}
\bibliography{multichannel}

%------------------------------------------------------------
\end{document}

%% file: defs.tex
\newtheorem{theorem}{Theorem}[section]
\newtheorem{lemma}[theorem]{Lemma}
\newtheorem{proposition}[theorem]{Proposition}
\newtheorem{corollary}[theorem]{Corollary}
\newenvironment{proof}{\noindent {\bf Proof} }{\endprf\par}
\def \endprf{\hfill {\vrule height6pt width6pt depth0pt}\medskip}

\definecolor{editcolor}{rgb}{1,0.87,0.87}
\makeatletter\newenvironment{editboxx}[1][.95]
{\begin{lrbox}{\@tempboxa}\begin{minipage}{#1\linewidth}\vspace{1ex}}
{\vspace{0ex}\end{minipage}\end{lrbox}\vspace{1ex}\fcolorbox{black}{editcolor}{\usebox{\@tempboxa}}\vspace{1ex}}
\makeatother

\newcommand{\R}{\mathbb{R}}
\newcommand{\C}{\mathbb{C}}

\renewcommand{\>}{\rangle}
\newcommand{\E}{\operatorname{E}}
\renewcommand{\P}{\operatorname{P}}

\newcommand{\SP}{\Gamma}

\newcommand{\iunit}{\mathrm{j}}
\newcommand{\conv}{\star}	% convolution
\newcommand{\Ns}{p}			% number of sources
\newcommand{\Nr}{q}			% number of receivers
\newcommand{\lin}{{\mathrm{lin}}}

%% file: Introduction.tex
This paper gives a theoretical treatment to the problem of channel
estimation in multiple-input multiple-output (MIMO) systems.  The
general scenario is illustrated in Figure~\ref{fig:mimo}: A set of
$\Ns$ {\em sources} emit different probe signals, which then travel
through different channels and are observed by $\Nr$ {\em receivers}.
We will assume that the channel between each source/receiver pair is
linear and time-invariant; if source $i$ sends probe signal $\phi_i$,
then receiver $j$ observes the convolution $\phi_i\conv h_{i,j}$ of
the probe signal with the corresponding channel response $h_{i,j}$.
The goal is to estimate all of the channel responses, and to do so
using the smallest total amount of probing time.

%----------------------------------------------------------------------------------
%\subsection{Problem formulation}

We will focus on the discrete version of this problem.  We assume that each channel response $h_{i,j}$ has length $n$.  If a single source $i$ emits a probe sequence $\{\phi_i(1),\phi_i(2),\ldots,\phi_i(m)\}$ of length $m$, receiver $j$ observes\footnote{We have take $m\geq n$ here, which does not affect the discussion too much at this point, but will be important later on.} the linear convolution
\begin{equation}
	\label{eq:linconv}
	y^\lin_{i,j} = \phi_i\star h_{i,j} = 
	\begin{bmatrix}
		\phi_i(1) & 0 & \hdots & 0\\
		\phi_i(2) & \phi_i(1) & \hdots & 0 \\
		\vdots & & & \\
		\phi_i(n) & \phi_i(n-1) & \hdots & \phi_i(1) \\
		\vdots & & & \\
		\phi_i(m) & \phi_i(m-1) & \hdots & \phi_i(m-n+1) \\
		0 & \phi_i(m) & \hdots & \phi_i(m-n+2) \\
		\vdots & & & \\
		0 & 0 & \hdots & \phi_i(m)
	\end{bmatrix}
	\begin{bmatrix}
		h_{i,j}(1) \\ h_{i,j}(2) \\ \vdots \\ h_{i,j}(n)
	\end{bmatrix} 
	=: \Phi^\lin_i h_{i,j}.
\end{equation}
Recovering $h_{i,j}$ from $y^\lin_{i,j}$ is a classical deconvolution problem.  The inverse problem can be made very well conditioned if $\phi_i$ is chosen carefully; if not, then the inversion can be regularized using some type of prior information about the channel.
  
We will measure the cost of the channel estimation by the amount of
time we spend probing the channel, which we can see is proportional to
$n+m-1 = O(n+m)$, the number of rows in the linear system in
\eqref{eq:linconv}.  From a single source, we can estimate the
response to all of the receivers by emitting a single probing sequence
and solving \eqref{eq:linconv} at each receiver $j$.  If there are
multiple sources, then typically the sources are activated
one-at-a-time. In this case, the total activation time required to
estimate all of the $h_{i,j}$ becomes $O(\Ns(n+m))$.  In theory, $m$
could be made as small as we like in this situation, giving us a lower
bound on the cost of $O(n\Ns)$.

In this paper, we propose and rigorously analyze an alternative
strategy for estimating the channels between each source-receiver
pair.  Our strategy will reduce the total amount of time spent on
probing the channels by activating all of the sources
simultaneously. (This approach was first proposed in the context 
of seismic imaging in a related conference paper \cite{neelamani08ef}.)  
Now, of course, the sources will
interfere with one another, and the receiver will observe a
combination of each source convolved with its respective channel.
With all $\Ns$ sources active, the observations at receiver $j$ can be
written as the following system of equations
\begin{equation}
	\label{eq:multilinconv}
	y^\lin_j = \sum_{i=1}^\Ns y^\lin_{i,j} = 
	\begin{bmatrix}
		\Phi^\lin_1 & \Phi^\lin_2 & \cdots & \Phi^\lin_\Ns 
	\end{bmatrix}
	\begin{bmatrix}
		h_{1,j} \\ h_{2,j} \\ \vdots \\ h_{\Ns,j}
	\end{bmatrix}
	=: \Phi^\lin h_j.
\end{equation}
The $\Phi^\lin$ in \eqref{eq:multilinconv} is the $(m+n-1)\times n\Ns$
matrix formed by concatenating the source convolution matrices
$\Phi^\lin_i$ row-wise; the $h_j$ is the unknown $n\Ns$-vector
consisting of the $\Ns$ channel responses for the path between each
source and receiver $j$.  With all of the sources activated
simultaneously, the total cost of the acquisition is $O(n+m)$, but now
the channel responses are interfering with one another.  The question
now is how long (quantified by $m$) the probing sequences must be to
reliably ``untangle'' the individual $h_{i,j}$ from $\Phi^\lin h_j$.
If the probing sequences are chosen carefully and in concert with one
another, the system in \eqref{eq:multilinconv} will be invertible for
$m\approx n\Ns$, again making the total activation time $O(n\Ns)$.  If
we are interested in recovering all possible channels without making
any assumptions about their structure, we of course cannot have
$m<np$.

We will show that if the combined channel response $h_j$ is {\em
sparse}, then the probing sequences can be significantly shorter than
$np$ if they are {\em random}.  This problem, along with the tools we
will deploy to solve it, is closely related to recent work in the
field of {\em compressive sensing} (CS).  The theory of CS basically
states that vectors $x_0$ with $s$ non-zero components can be
recovered from an underdetermined set of linear measurements $y=\Phi
x_0$ if the the matrix $\Phi$ is sufficiently diverse (the precise
technical conditions are reviewed in detail in
Section~\ref{sec:sparse-recovery}).  The essential contribution of
this paper is to show that when the sequences $\{\phi_i(t)\}_t$
consist of independent and identically distributed Gaussian random
variables, the matrix $\Phi^\lin$ in \eqref{eq:multilinconv} meets
this criterion for pulse lengths $m$ that are within a
poly-logarithmic factor of the sparsity $s$.  In particular,
Theorems~\ref{th:Emain} and \ref{th:Pmain} combined with
Proposition~\ref{prop:stable-recovery} shows that if the total number
of significant components in $h_j$ is $s$, then it can be recovered
from $y^\lin$ for
\[
	m ~\lesssim~ s\cdot\log^5(n\Ns), 
\]
reducing the total time the sources are activated to $O(n +
s\log^5(n\Ns))$.  When the channels are sparse, that is $s \ll n\Ns$, then
the cost of acquiring all of the channels is not much
more than acquiring a single channel independently.
While having the sources activated simultaneously introduces
``cross-talk'' between the different channels, the use of different
random codes by each source coupled with the sparse structure of the
channels allows us to separate the cross-talk into its constituent
components.

In the remainder of this section, we will discuss some applications of
the channel separation problem and review recent related work.
Section~\ref{sec:sparse-recovery} provides an overview of sparse
recovery from underdetermined linear measurements.
Section~\ref{sec:separation-theorem} carefully states our main
theorems, which provide a sufficient lower bound on the length of the
probing signals (in relation to the number of sources, the length of
the channels, and their sparsity) that allows us to robustly recover
$h_j$ from $y^\lin$ using any number of sparse recovery algorithms.
Proof of these theorems is given in Sections~\ref{sec:proofmean} and
\ref{sec:prooftail}.  The proofs rely heavily on estimates for random
sums of rank-1 matrices, which are overviewed in the Appendix.

%-------------
\begin{figure}
	\centerline{
	\begin{tabular}{ccc}
		\includegraphics[height=1.6in]{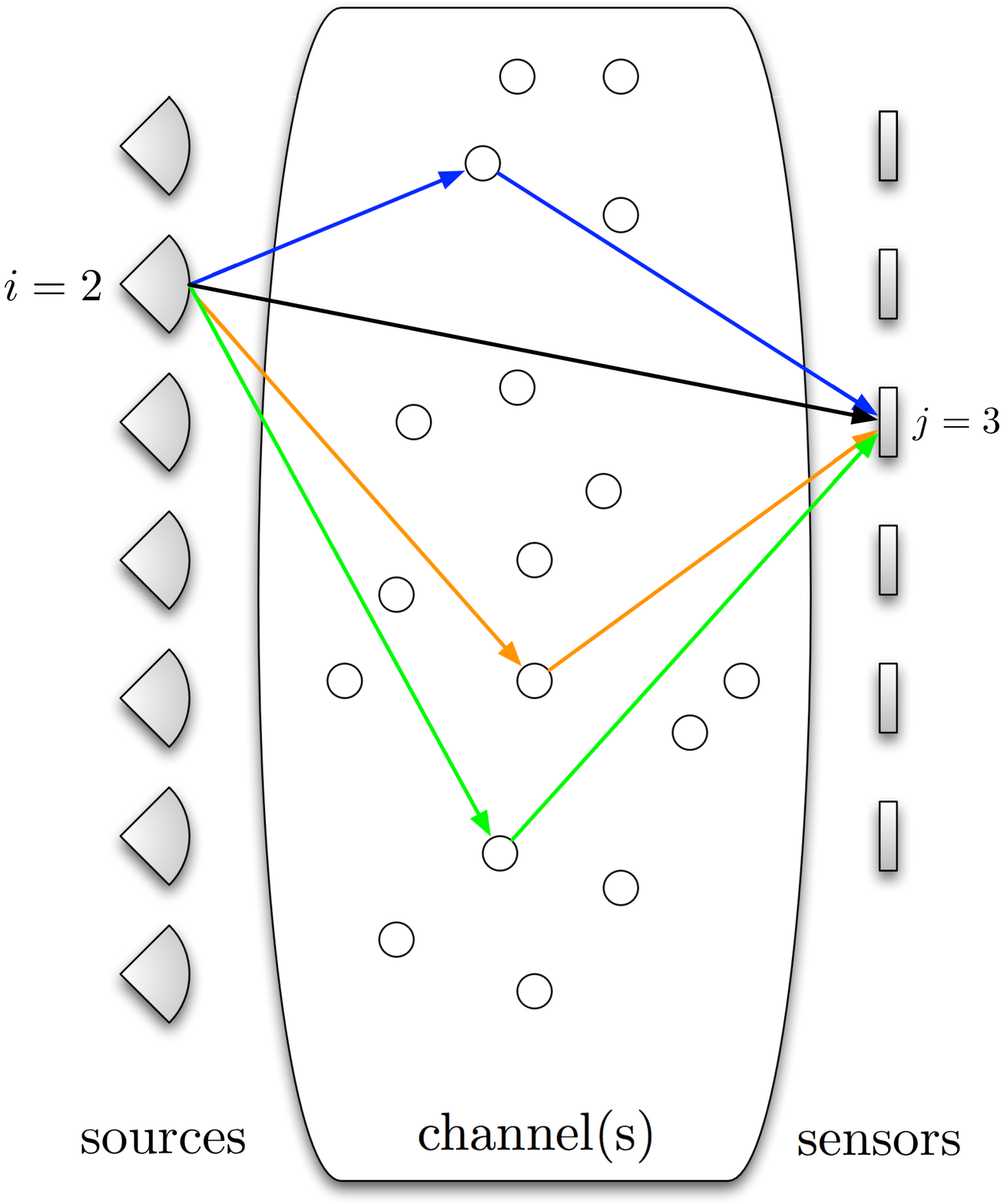} & \hspace{.5in} &
		\includegraphics[height=1.2in]{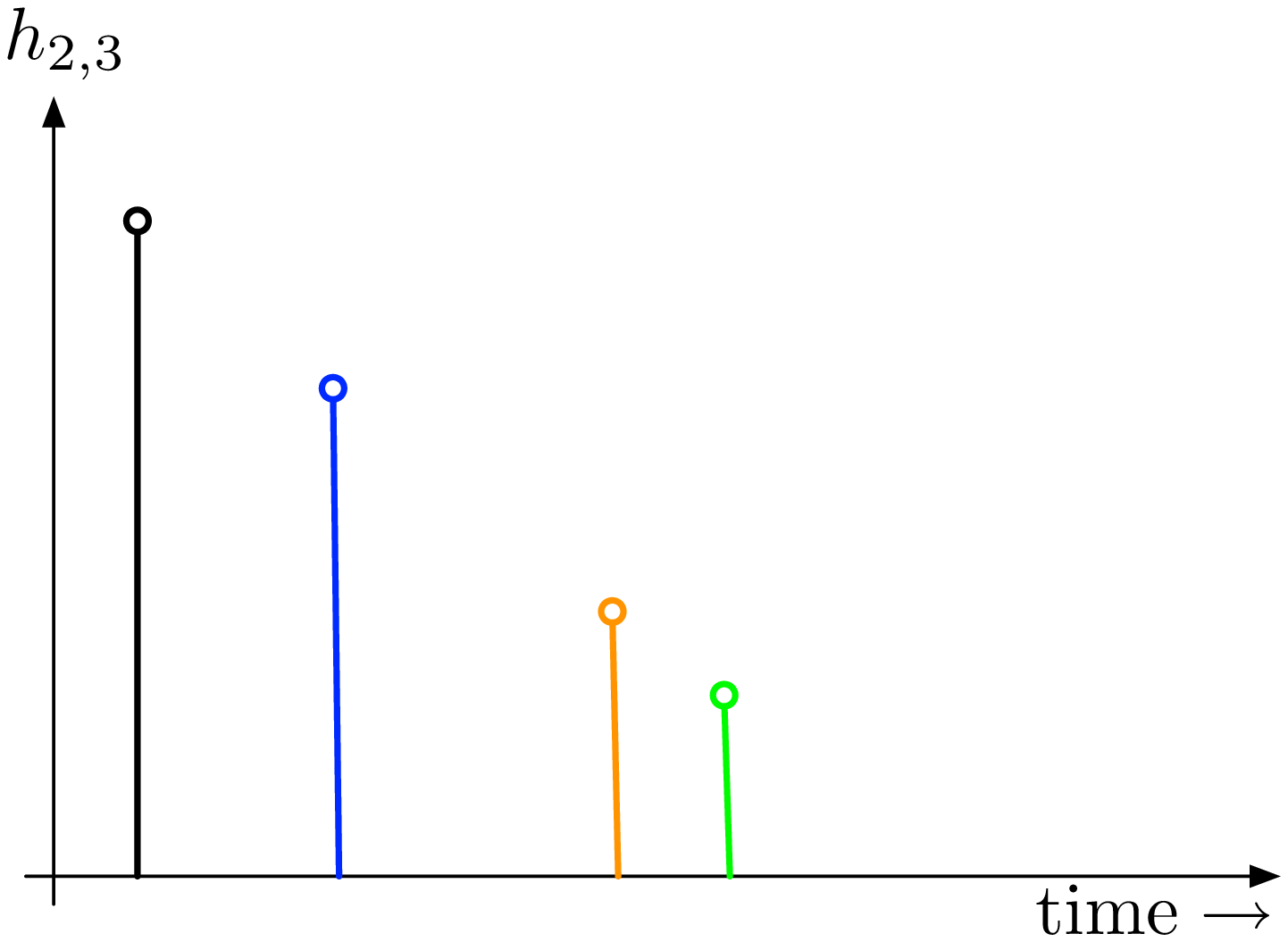} \\
		(a) & & (b)
	\end{tabular}
	}
	\caption{\small\sl (a) Multiple sources on
	left, multiple receivers on the right. The black arrow denotes
	the direct arrival from source $2$ to receiver $3$. The colored arrows denote
	the indirect arrivals caused by the different reflectors
	(denoted by small circles).   
	(b) Channel response $h_{2,3}$ between source $2$ and receiver $3$.}
	\label{fig:mimo}
\end{figure}
%-------------

%----------------------------------------------------------------------------------
\subsection{Applications}

For further motivation, we discuss three specific scenarios in which this multichannel separation problem arises.

{\bf Seismic exploration and forward modeling.} Subsurface images of
the earth are formed by activating different points on the surface
with acoustic sources, then measuring the response at a number of
receivers locations.  From these recorded responses, a 3D subsurface
model of the earth (consisting of the local velocity of the
propagating elastic waves) can be reconstructed using what is known as
full waveform inversion (FWI).  Dense samplings for the positions of
the acoustic sources lead to higher resolution reconstructions, but
also longer field acquisitions and more computationally intensive
inversion.
	
The theoretical results in this paper suggest that these expenses can
mitigated by activating the sources simultaneously using different
random patterns.  In the field, this reduces the amount of time
required for the acquisition.  Although the sources will interfere
with one another, the individual responses can be separated afterwards
by taking advantage of the sparsity of each of the channel
responses\footnote{Sparse models are common in seismic imaging
\cite{claerbout:826,scales1990}.  In practice, additional gains are
realized by going beyond the setting treated in this paper and
modeling the channels jointly, viewing them as cross sections of a
larger 3D image (with an associated sparse transform) rather than as
individual sparse channels; see \cite{herrmann:GJI}, for example.}.
The source waveforms will have to be longer than if each of the $\Ns$
sources were activated individually, but the net activation time
across all sources is much smaller than $\Ns$ individual channel
probes.
	
Sparse channel separation can also reduce the amount of computation
required for the inversion.  The most expensive step in wavefront
inversion is testing a candidate model to see how well it fits the
measurements that have been collected.  This so-called {\em forward
modeling} simulation consists of solving an extremely large PDE.  The
cost of this simulation is proportional to the length of the source
signals (i.e.\ the number of time steps required), but does not depend
at all on the number of sources that are active at one time ---
running a simulation with a single source active costs takes just as
long per time step as with many sources active.  If we simulate each
of the $\Ns$ sources individually, we will need to run each simulation
for $O(n)$ time for a total cost of $O(n\Ns)$ time steps (and the cost
of each time step can be extremely high).  If we simulate the sources
simultaneously, then the number of time steps in the simulation can be
$O(n + s\log^5 (n\Ns))$.  Given the results of the simulation with
simultaneous active sources, we will of course have to recover the
individual channel responses using some type of sparse recovery
algorithm (solving the optimization program in \eqref{eq:l1qc} below,
for example).  But the computations required for this recovery are
minor in comparison to the forward modeling simulation, especially
given recent progress in optimization algorithms
\cite{figueiredo07gr,hale08fi,yin08br,becker09ne} and the fact that
the system $\Phi^\lin$ can be applied quickly using FFTs.
	
	Source separation for seismic exploration is explored in further detail in the companion publications \cite{neelamani08ef,neelamani09ef}.

{\bf Channel estimation in MIMO communications.}  When information is transmitted wirelessly, it is often the case that reflections cause there to be multiple paths from the source to the receiver.  Instead of the transmitted waveform, the receiver observes a convolution of this waveform with a channel response --- if the number of reflectors is small, this response is sparse.  To compensate for this multipath effect, the channel is periodically estimated by having the source emit a known probing sequence that the receiver can subsequently deconvolve.  If there are multiple transmitters and multiple receivers, we can save time by probing all of the channel pairs simultaneously, and separating the individual responses using sparse recovery.  This approach is particularly useful when the channel is changing rapidly, a common problem in underwater acoustic communications \cite{catipovic90pe}.

{\bf Coded aperture imaging.}  In \cite{marcia08fa,marcia08su,marcia09co}, an imaging architecture is introduced to increase the field-of-view (FOV) of a camera using {\em coded apertures}.  A coded aperture is a series of small openings (apertures) whose net effect is to convolve a target image with a sequence (code) determined by the pattern in which these openings appear.  Coded apertures offer a way around the classical trade-off between aperture size and image brightness; the multiple apertures overlay many copies of the image at slightly different shifts, making the image incident on the detector array ``bright'' and easily recovered (via deconvolution) if the aperture code is chosen carefully (e.g.\ the MURA patterns in \cite{gottesman89ne}).
		
	The essential idea from \cite{marcia09co} to increase the FOV without sacrificing resolution is illustrated in Figure~\ref{fig:fovca}: the image is broken into $\Ns$ subimages, each of which we would like to recover to a resolution of $n$ pixels.  Rather than measuring each subimage directly, which would require a detector array of size $m=n\Ns$, we pass each image through its own coded aperture of size $m$ and these coded subimages are combined onto a detector array of size $m$.  The task at hand, then, is to recover the full $n\Ns$ pixel image from these $m$ measurements.  
	
	This problem also conforms to our multichannel framework\footnote{Passing an image through a coded aperture has the effect of convolving it with a binary code.  The theoretical results presented in this paper require the code to be Gaussian; this requirement was imposed so that each convolution could be diagonalized in the Fourier domain, which allows us to apply recent results from the theory of random matrices to prove Theorems~\ref{th:Emain} and \ref{th:Pmain}.  In practice, we would expect there to be little difference between then Gaussian and binary cases.} --- in this case, we have $\Ns$ sources and one receiver.  Here, the known ``probe signals'' are the coded aperture patterns and the unknown channels $h_{i,1}$ are the different subimages.  The main results of this paper say that if the entire image is approximately $s$ sparse, than the size $m$ of the detector array needs only to be only on the order of $s$ (within a log factor) rather than the full resolution $n\Ns$.  If the images we are reconstructing are consecutive frames in a video sequence, the image sparsity can come from looking at the differences between consecutive frames.

%--------------	
\begin{figure}
	\centering
	\includegraphics[height=1.5in]{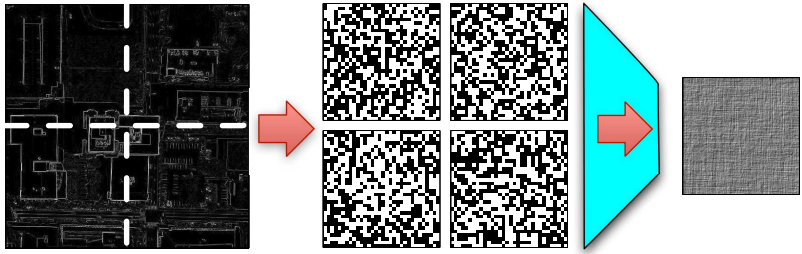}
	\caption{\small\sl Sketch of the architecture proposed in \cite{marcia09co} for increasing the field of view of a camera.  The image is broken into $p$ subimages, each subimage passes through a different coded aperture, and the coded subimages are combined on the detector array.  The effect of the coded aperture is to convolve the subimage with an associated code, and so \eqref{eq:multilinconv} models the process mathematically.}
	\label{fig:fovca}
\end{figure}
%--------------	

%----------------------------------------------------------------------------------
\subsection{Relationship to previous work}

We have cast the multichannel separation problem as recovering a
sparse vector from an underdetermined, random system of equations.
This general problem has been studied extensively over the past five
years under the name of compressive sensing (CS).  The essential
results from this field state that if we observe $y=\Phi x_0$, where
$x_0$ is $s$-sparse and $\Phi$ is an $m\times N$ matrix that obeys a
technical condition called the {\em restricted isometry property} (see
\eqref{eq:uup} below), then $x_0$ can be recovered from $y$ even when
$m \ll N$, and this recovery is stable in the presence of measurement
noise, and robust against modeling error (i.e.\ it is effective even
when $x_0$ is not perfectly sparse)
\cite{candes06ne,candes06st,donoho06co}.  Random matrices that obey
this property for $m$ within a logarithmic factor of $s$ include
matrices with independent entries \cite{candes06ne,baraniuk08si},
matrices that have been subsampled from orthobases consisting of
vectors whose energy is almost evenly distributed between their
entries \cite{rudelson08sp}, as well as other matrices with more
structured randomness \cite{tropp10be,romberg09co}.  These measurement
systems provide efficient encodings for $x_0$ because the number of
measurements we need to make is roughly proportional to the number
active elements.  The results from
Section~\ref{sec:separation-theorem} tell us that matrices formed by
concatenating a series of $p$ random convolutions are another such
efficient encoding (with $N=np$).

Early work on signal processing algorithms using sparse models for channel estimation can be found in \cite{cotter02sp} and \cite{fuchs99mu}.  In \cite{haupt08to}, estimation of a single channel using a pulse consisting of a sequence of independent Gaussian random variables is explored; the mathematical results of \cite{haupt08to} are framed in the language of CS, and the key recovery condition (the restricted isometry property in \eqref{eq:uup} below) is established for pulse lengths of $O(n+s^2\log n)$.  The paper \cite{rauhut09ci} show that the recovery conditions can be improved to $O(n + s\log^3 n)$ when the observations are noiseless and the channel is exactly $s$-sparse.  Using convolution with a random pulse to perform compressive sensing was also considered in the context of imaging in \cite{romberg09co} and as a way to handle streaming data in \cite{tropp06fi}.  Results for super-resolved radar imaging using ideas from CS can be found in \cite{herman09hi}.  In this paper, the undetermined system arises not because we are subsampling a signal after it has been convolved with a pulse, but by combining the convolutions from multiple channels into one observed sequence.

Multichannel separation also bears some resemblance to the problem of finding the sparsest decomposition in a union of bases \cite{donoho01un,gribonval03sp,candes06qu,tropp08ra}; this resemblance becomes even more pronounced when we recast the problem using circular convolution (see Section~\ref{sec:lintocirc}) and take $m=n$.  We can think of each convolution matrix as a different basis, and search for a way to write the measurements as a superposition of a small number of vectors chosen from these bases.  In contrast to previous work on this problem, the bases here are random and not quite orthogonal (the related paper \cite{romberg09mu} considers an alternative way to generate the random pulses so that each of the convolution matrices is exactly orthogonal).

%% file: SparseRecovery.tex
In the previous section, we set up multiple channel estimation as a linear inverse problem.  Classically, these types of problems are solved using least-squares; the stability of the solution is almost completely characterized through the eigenvalues of $\Phi^*\Phi$.  If for all $x\in\R^{np}$ we have
\begin{equation}
	\label{eq:framebounds}
	(1-\delta)\|x\|^2_2 ~\leq~ \|\Phi x\|^2_2 ~\leq~ (1+\delta)\|x\|^2_2
\end{equation}
for some small $0\leq\delta < 1$, then recovering $x$ from $\Phi x$ is well-posed and stable in the presence of noise.  Of course, $m < np$, then the system will be underdetermined, $\Phi$ has a nonzero nullspace, and the lower-bound in \eqref{eq:framebounds} cannot hold.  It appears that to simultaneously estimate all of the channel responses, the length $m$ of the probe sequence must exceed $np$.  

Recent results from compressive sensing have told us that if the vector we are trying to recover is sparse, then a much weaker condition on $\Phi$ is sufficient for well-posed, stable recovery.  In particular, if \eqref{eq:framebounds} holds for all $2s$-sparse vectors $x$, rather than all $x\in\R^{np}$, then we will be able to recover $h_j$ from $r_j=\Phi h_j$ about as well as if we had observed the $s$ largest (most important) entries in $h_j$ directly.   

We can make this precise in the following manner.  Denote by $B^\SP_2$ the set of all vectors $x\in\R^{np}$ that are nonzero only on the set $\SP\subset\{1,\ldots,np\}$ and have unit $\ell_2$ norm.  For a square matrix $A$, we define the $\|\cdot\|_s$ norm as 
\begin{equation}
	\label{eq:snorm}
	\|A\|_s = \sup_{\substack{|\SP|\leq s\\ x,y\in B^\SP_2}} |y^*Ax|,
\end{equation}
where the supremum is taken over all $s$-sparse vectors with unit energy.  (We use $*$ for the transpose of a real-valued vector or matrix, or conjugate-transpose for a complex-valued vector or matrix.)  It is easy to see that if
\[
	\|I-\Phi^*\Phi\|_s \leq \delta_s
\]
then
\begin{equation}
	\label{eq:uup}
	(1-\delta_s)\|x\|^2_2~\leq~\|\Phi x\|^2_2~\leq~(1+\delta_s)\|x\|^2_2
	\quad\text{for all $s$-sparse $x$}.
\end{equation}

Establishing \eqref{eq:uup}, which has gone by the names {\em uniform uncertainty principle} and {\em restricted isometry property} in the CS literature, is the key for stable sparse recovery \cite{candes06ne,candes06st,candes07da,needell09co,blumensath08it}.  
The following proposition gives us a concrete algorithm for recovering a sparse vector from measurements made by a matrix that satisfies \eqref{eq:uup}.
\begin{proposition}[\cite{candes06st}]
	\label{prop:stable-recovery}
	Let $x_0$ be an $s$-sparse vector, and $\Phi$ be a matrix that obeys \eqref{eq:uup} with
	$\delta_{C_1s}\leq C_2$, where $C_1\geq 1$ and $0<C_2<1$ are constants.  Given noisy observations 
	$y = \Phi x_0 + e$, where $e$ is an error vector with norm at most $\|e\|_2\leq\epsilon$, the solution
	$\tilde{x_0}$ to the optimization program
	\begin{equation}
		\label{eq:l1qc}
		\min_x~\|x\|_1 \quad\text{subject to}\quad \|\Phi x - y\|_2\leq\epsilon
	\end{equation}
	will obey 
	\[
		\|\tilde{x}_0 - x_0\|_2 ~\leq~ C_3\epsilon,
	\]
	where $C_3$ is a known universal constant.  In addition, if $x_0$ is a general non-sparse vector, 
	then the solution to \eqref{eq:l1qc} obeys
	\[
		\|\tilde{x}_0 - x_0\|_2 ~\leq~ C_4\left(\epsilon + s^{-1/2}\|x_0 - x_{0,s}\|_1 \right),
	\]
	where $x_{0,s}$ is the best $s$-sparse approximation to $x_0$; the nonzero components in $x_{0,s}$ 
	are the $s$ largest components of $x_0$.
\end{proposition}
The constants in the theorem above are known to be small.  For example, in \cite{candes08re} it is shown that we need $C_1=2$ and $C_2\leq\sqrt{2}-1$, and with $C_2=1/4$, we have $C_4\leq 6$.
Similar stability results hold for recovery procedures other than $\ell_1$ minimziation.  In particular, in \cite{needell09co} and \cite{blumensath08it}, it is shown that particular types of iterative thresholding algorithms can achieve essentially the same performance after a very reasonable number of iterations.  

The main result of this paper, codified in Theorems~\ref{th:Emain} and \ref{th:Pmain} below, is that the $\Phi$ which arises in the multichannel separation problem will obey the restricted isometry property \eqref{eq:uup} for $s$ almost proportional (within a $\log$ factor) to $m$.

%% file: SeparationTheorem.tex
%------------------------------------------------------------------------------------------
\subsection{From linear to circular convolution}
\label{sec:lintocirc}

Rather than analyze the spectral properties of $\Phi^\lin$ in \eqref{eq:multilinconv} directly, we will replace it with a slightly modified version whose components are submatrices of large circular matrices, and thus can be diagonalized in the Fourier domain, which simplifies the analysis considerably.  To do this, we simply ``pre-process'' the measurements by adding some of them together to create a slightly shorter observation vector.

To start, consider the single source measurements $y^\lin_{i,j}$ in equation \eqref{eq:linconv}, with the pulse length $m$ exceeding the length of the channel $n$.  Suppose that we add the first $n-1$ measurements to the last $n-1$ measurements to form
\begin{align*}
	y_{i,j} &= \begin{bmatrix} 
		y^\lin_{i,j}(n) \\ \vdots \\ y^\lin_{i,j}(m) \\ 
		y^\lin_{i,j}(m+1) + y^\lin_{i,j}(1) \\ \vdots \\
		y^\lin_{i,j}(m+n-1) + y^\lin_{i,j}(n-1)
	\end{bmatrix} \\[4mm]
	& = \begin{bmatrix}
		\phi_i(n) & \phi_i(n-1) & \hdots & \phi_i(1) \\
		\vdots & & & \\
		\phi_i(m) & \phi_i(m-1) & \hdots & \phi_i(m-n+1) \\
		\phi_i(1) & \phi_i(m) & \hdots & \phi_i(m-n+2) \\
		\vdots & & & \\
		\phi_i(n-1) & \phi_i(n-2) & \hdots & \phi_i(m)
	\end{bmatrix}
	\begin{bmatrix}
		h_{i,j}(1) \\ h_{i,j}(2) \\ \vdots \\ h_{i,j}(n)
	\end{bmatrix} \\[4mm] 
	& =: \Phi_i h_{i,j}
\end{align*}
The matrix $\Phi_i$ consists of the first $n$ columns of an $m\times
m$ circulant matrix with 
\[
r_i^* = [\phi_i(n)~\cdots~\phi_i(1)~\phi_i(m)~\cdots~\phi_i(n+1)],
\]
as the first row.  As such, we can use the discrete Fourier transform
to diagonalize $\Phi_i$.  Let $F$ be the $m\times m$ normalized
discrete Fourier matrix with entries
\[
	F(\omega,t) = \frac{1}{\sqrt{m}}e^{-\iunit 2\pi(\omega-1)(t-1)/m},
\]
and let $F_{(1:n)}$ denote the $m\times n$ matrix consisting of the first $n$ columns of $F$.  Then 
\begin{equation}
	\label{eq:Phiif}
	\Phi_i = F^*G_iF_{(1:n)}, \quad\text{with}\quad G_i=\operatorname{diag}(\{g_i(\omega)\}_{\omega=1}^m).
\end{equation}
The vector $g_i(\omega)$ is the (re-normalized) Fourier transform of $r_i$:
\[
	g_i=m\cdot Fr_i .
\]

When all the sources are active simultaneously, we can perform the same manipulations on the composite linear system \eqref{eq:multilinconv}, combining the first $n-1$ entries in $y^\lin_j$ with the last $n-1$ to yield
\begin{equation}
	\label{eq:yPhih}
	y_j = 
	\begin{bmatrix}
		\Phi_1 & \Phi_2 & \cdots & \Phi_\Ns
	\end{bmatrix}
	\begin{bmatrix}
		h_{1,j} \\ h_{2,j} \\ \vdots \\ h_{\Ns,j}
	\end{bmatrix}
	=: \Phi h_j.
\end{equation}
As in \eqref{eq:Phiif}, we can write $\Phi$ as
\begin{equation}
	\label{eq:phimulti}
	\Phi = F^*\begin{bmatrix}G_1F_{(1:n)} & G_2F_{(1:n)} & \cdots & G_\Ns F_{(1:n)}  \end{bmatrix}.
\end{equation}

We assume that each source emits an independent random waveform.  That is, we take the probe samples $\{\phi_i(t)\}_{i,t}$ to be iid Gaussian random variables with zero mean and variance $m^{-1}$ (so each probing waveform $\phi_i$ has unit energy in expectation).  Since the $\phi_i(t)$ are iid Gaussian, the corresponding Fourier transforms $g_i(\omega)$ are sequences of conjugate symmetric complex-valued Gaussian random variables:
\[
	g_i(\omega)\sim\begin{cases} \mathrm{Normal}(0,1) & \omega = 0,m/2+1 \\
	\mathrm{Normal}(0,1/\sqrt{2}) + \iunit\cdot\mathrm{Normal}(0,1/\sqrt{2}) & \omega=2,\ldots,m/2
	\end{cases}
\]
and $g_i(\omega) = g_i(n-\omega+2)^*$ for $\omega=n/2+2,\ldots,n$.

%------------------------------------------------------------------------------------------
\subsection{Recovery theorems}

Our main theorem shows that the random matrix $\Phi$, generated from the random sequences $\{g_i(\omega)\}$ as in \eqref{eq:yPhih}, is an approximate restricted isometry in expectation for $m\gtrsim s\log^5(n\Ns)$.

\begin{theorem}
	\label{th:Emain}
	Let $\Phi$ be as in \eqref{eq:phimulti}.  There exists constants $C_5$ and $C_6$ such that
	\begin{equation}
		\label{eq:Eratio}
		\E\|I-\Phi^*\Phi\|_s ~\leq~ \sqrt{\frac{C_5\cdot s\cdot \log^2 s\log^2(m\Ns) \log(n\Ns)}{m}},
	\end{equation}
	when
	\[
		m\geq C_6\cdot s\cdot\log^2s\log^2(m\Ns)\log(n\Ns).
	\]
\end{theorem}

It is straightforward to turn Theorem~\ref{th:Emain} into a direct statement about the restricted isometry constants.
\begin{corollary}
	There is a constant $C_7$ such that
	\[
		\E\|I-\Phi^*\Phi\|_s ~\leq~\delta_s
	\]
	when
	\begin{equation}
		\label{eq:m-mean}
		m\geq C_7\delta_s^{-2}\cdot s\cdot\log^5(n\Ns),
	\end{equation}
	for any $0< \delta_s\leq 1$, provided that $m \leq n\Ns$.
\end{corollary}
To see how \eqref{eq:m-mean} follows from \eqref{eq:Eratio}, notice that for $m\leq n\Ns$ and $s\leq n\Ns$, we have
\[
	\log(m\Ns) ~=~ \log m + \log \Ns ~\leq~2\log(n\Ns)
\]
and so
\[
	\sqrt{\frac{C_5\cdot s\cdot \log^2 s\log^2(m\Ns) \log(n\Ns)}{m}} ~\leq~ 
	\sqrt{\frac{4C_5\cdot s\cdot \log^5(n\Ns)}{m}},
\]
and we can choose $m$ as in \eqref{eq:m-mean}. 

Theorem~\ref{th:Emain} gives us a lower bound on the length of a pulse sufficient to endow, in expectation, $\Phi$ with certain restricted isometry constants.  The following theorem gives us a lower bound for the length of the pulses that guarantees that $\Phi$ has certain isometry constants with high probability. 

\begin{theorem}
	\label{th:Pmain}
	Let $\Phi$ and $\delta_s$ be as in Theorem~\ref{th:Emain}.  There exists constants $C_8$ and $C_9$
	such that 
	\[
		\P\left\{\|I-\Phi^*\Phi\|_s > \delta_s\right\} ~\leq~ C_8(n\Ns)^{-1}
	\]
	when
	\begin{equation}
		\label{eq:m-tail}
		m \geq C_9\cdot \delta_s^{-2}\cdot s\cdot\log^6(n\Ns).
	\end{equation}
\end{theorem}
It is worth mentioning that we chose a probability of failure of $\sim (np)^{-1}$ mostly out of convenience.  In fact, the probability can be made arbitrarily small by adjusting the constant $C_9$; we could achieve a failure rate of $(np)^{-\alpha}$ for any $\alpha>1$ by making the constant in \eqref{eq:m-tail} $C_9\alpha$.

The essential consequence of the next theorem is that for pulse lengths \eqref{eq:m-tail}, we can simultaneously estimate the channel responses from all sources to receiver $j$, which are concatenated in the vector $h_j$, from either concatenated circular convolution observations $\Phi h_j$ or concatenated linear convolution observations $\Phi^\lin h_j$.  As linear convolution observations are more typical, we state our channel separation corollary in terms of $\Phi^\lin$.  
\begin{corollary}
	Suppose we observe 
	\[
		y^\lin_j = \Phi^\lin h_j + e,
	\]	
	where $y^\lin, \Phi^\lin$, and $h_j$ are as in \eqref{eq:multilinconv} and $e$ is an unknown vector of measurement errors with $\|e\|_2\leq\epsilon$.  Take $C_1,C_2,$ and $C_4$ as in Proposition~\ref{prop:stable-recovery}, and take $m$ as in Theorem~\ref{th:Pmain} so that  $\delta_{C_1s}\leq C_2$, where $\delta_{C_1s}$ is the isometry constant for the concatenated circulant matrix $\Phi$ generated from $\Phi^\lin$ as in \eqref{eq:yPhih}.  Then the solution $\tilde{h}_j$ to
	\[
		\min_h~\|h\|_1 \quad\text{subject to}\quad \|\Phi^\lin h - y^\lin_j\|_2\leq\epsilon.
	\]	
	is a close approximation to $h_j$ in that
	\begin{equation}
		\label{eq:linstable}
		\|\tilde{h}_j-h_k\|_2 ~\leq~ C_4\left(\sqrt{2}\epsilon + s^{-1/2}\|h_j-h_{j,s}\|_1\right),
	\end{equation}
	where $h_{j,s}$ is the best $s$-term approximation to $h_j$.
\end{corollary}

\begin{proof}
Theorem~\ref{th:Pmain} coupled with Proposition~\ref{prop:stable-recovery} give us robust reconstruction for observations made through the concatenated circulant system $\Phi$.  To establish the Proposition, we will make a concrete connection between the solutions to the linear and circular convolution inverse problems.

First, we consider the case where there is no noise and $h_j$ is perfectly $s$-sparse.  Given the circular observations $y = \Phi h_j$, we could solve \eqref{eq:l1qc} with $\epsilon=0$, making the constraints $\Phi x = y$.   With $m$ as in \eqref{eq:m-mean}, the solution $\tilde{h}_j$ will be exactly $h_j$ with high probability.  Stated differently, there is no vector in the nullspace of $\Phi$ that can be added to $h_j$ that lowers the $\ell_1$ norm.  Since $\operatorname{Null}(\Phi^\lin)\subset\operatorname{Null}(\Phi)$, we could also solve \eqref{eq:l1qc} with $y^\lin$ and $\Phi^\lin$ in place of $y$ and $\Phi$ and recover the signal exactly.

To make the connection when there is noise,  we use the following proposition, which is contained in \cite{candes06st,candes08re}, but is slightly stronger than Proposition~\ref{prop:stable-recovery}.
\begin{proposition}
	Under the conditions of Proposition~\ref{prop:stable-recovery}, if $d$ is {\em any} vector that satisfies $\|x_0+d\|_1 \leq \|x_0\|_1$ and $\|\Phi d\|_2\leq 2\epsilon$ (both of which must be true for $d=\tilde{x}_0-x_0$), then
\[
	\|d\|_2 ~\leq~ C\cdot \left(\epsilon + s^{-1/2}\|x_0-x_{0,s}\|_1\right).
\]
\end{proposition}
Now suppose we solve \eqref{eq:l1qc} with observations $y^\lin_j$ and matrix $\Phi^\lin$, denoting the solution $\tilde{h}^\lin_j$ and set $d=\tilde{h}^\lin_j - h_j$.  Since $h_j$ is feasible, we will have both $\|h_j+d\|_1\leq \|h_j\|_1$ and $\|\Phi^\lin d\|_2\leq 2\epsilon$.  We can write $\Phi = A\Phi^\lin$, where $A$ combines the first and last $n-1$ elements of a vector.  Since the maximum singular value of $A$ is $\sqrt{2}$, we also have
\[
	\|\Phi d\|_2 = \|A\Phi^\lin d\|_2 \leq \sqrt{2}\|\Phi^\lin d\|_2 \leq 2\sqrt{2}\epsilon.
\]
Thus the solution $\tilde{h}^\lin_j$ is as accurate as solving \eqref{eq:l1qc} with the circulant observations $y_j$ and matrix $\Phi$ with $\epsilon$ increased by a factor of $\sqrt{2}$.  Thus $\tilde{h}_j$ will obey \eqref{eq:linstable}.

\end{proof}

%% file: ProofMean.tex
The essential tool for establishing \eqref{eq:Eratio} is a variation
(Lemma~\ref{lm:rvnon}) of a lemma due to Rudelson and Vershynin
(Lemma~\ref{lm:rvmean}).  Most of our efforts will go towards
manipulating $I-\Phi^*\Phi$ to put it in a form where we can apply
Lemma~\ref{lm:rvnon}.  The basic flow of the proof is to divide
$I-\Phi^*\Phi$ into several parts, each of which can be written as a
sum of independent rank-1 matrices, and then apply the bounds in the
Appendix to each part.  This process is not difficult, but it is
somewhat laborious.  To aid the exposition, we have divided the proof
into steps, each one of which accomplishes a particular task.

We will not track the constants.  We will use $C$ to denote a constant that is independent of all the variables of interest ($s,n,m,\Ns$); the particular value of $C$ may change between instantiations.  We will give a constant a label in the subscript if we want to refer to it later.
  
To start, we set $Z=I-\Phi^*\Phi$.
\renewcommand{\theenumi}{{\bf E\arabic{enumi}}}
\begin{enumerate}

%
%--------------------------------------------------------------
%
\item\label{step:Erank1} {\bf Write $Z$ as a sum of rank-1 matrices.}  
Recall from \eqref{eq:phimulti} that we can write the multichannel convolution matrix $\Phi$ as
\[
	\Phi = F^*\begin{bmatrix}G_1F_{(1:n)} & G_2F_{(1:n)} & \cdots & G_pF_{(1:n)}  \end{bmatrix}
\] 
where the $G_k$ are diagonal matrices consisting of the re-normalized Fourier transforms of the sources.  We can write $\Phi^*\Phi$ in matrix form as
\begin{align*}
	\Phi^*\Phi &= \begin{bmatrix} F^*_{(1:n)}G_1^* \\  F^*_{(1:n)}G_2^* \\ \vdots \\
	F^*_{(1:n)}G_p^*\end{bmatrix}
	FF^*
	\begin{bmatrix} G_1F_{(1:n)} & G_2F_{(1:n)} & \cdots & G_pF_{(1:n)} \end{bmatrix} \\
	%&= \begin{bmatrix}
	%	F^*_{(1:n)}G_1^*G_1F_{(1:n)} & F^*_{(1:n)}G_1^*G_2F_{(1:n)} & 
	%	\cdots & F^*_{(1:n)}G_1^*G_pF_{(1:n)} \\
	%	F^*_{(1:n)}G_2^*G_1F_{(1:n)} & F^*_{(1:n)}G_2^*G_2F_{(1:n)} & 
	%	\cdots & F^*_{(1:n)}G_2^*G_pF_{(1:n)} \\
	%	\vdots & \vdots & & \vdots \\
	%	F^*_{(1:n)}G_p^*G_1F_{(1:n)} & F^*_{(1:n)}G_p^*G_2F_{(1:n)} & 
	%	\cdots & F^*_{(1:n)}G_p^*G_pF_{(1:n)}
	%\end{bmatrix} \\
	&=
	\begin{bmatrix}
		F^*_{(1:n)} & & & \\
		 & F^*_{(1:n)} & &  \\
		& & \ddots & \\
		& & & F^*_{(1:n)}
	\end{bmatrix}
	\begin{bmatrix}
		G_1^*G_1 & G_1^*G_2 & \cdots & G_1^*G_p \\
		G_2^*G_1 & G_2^*G_2 & \cdots & G_2^*G_p \\
		\vdots & \vdots & & \vdots \\
		G_p^*G_1 & G_p^*G_2 & \cdots & G_p^*G_p
	\end{bmatrix}
	\begin{bmatrix}
		F_{(1:n)} & & & \\
		 & F_{(1:n)} & &  \\
		& & \ddots & \\
		& & & F_{(1:n)}
	\end{bmatrix},
\end{align*}
where we have used the fact that $FF^*=I$.  We can compact this expression by introducing $f_{k,\omega}\in\C^{np}$ as the vector which has column $\omega$ of $F^*_{(1:n)}$ in entries $(k-1)n+1,\ldots,kn$ and is zero elsewhere.  Then we can rewrite $\Phi^*\Phi$ as 
\begin{equation}
	\label{eq:PhitPhi}
	\Phi^*\Phi = \sum_{k=1}^p\sum_{j=1}^p\sum_{\omega=1}^m 
	g_{k}(\omega)^*g_j(\omega)f_{k,\omega}f_{j,\omega}^*.
\end{equation}
Since
\begin{equation}
	\label{eq:diagI}
	\sum_{k=1}^p\sum_{\omega=1}^m f_{k,\omega}f_{k,\omega}^* = I.
\end{equation}
we can now write $Z = I - \Phi^*\Phi$ as 
\begin{align}
	\label{eq:h1h2}
	Z &= \sum_{k=1}^p\sum_{\omega=1}^m 
	(1-|g_k(\omega)|^2)f_{k,\omega}f_{k,\omega}^* + 
	\sum_{j\not= k}\sum_{\omega=1}^m g_k(\omega)^*g_j(\omega)
	f_{k,\omega} f_{j,\omega}^* \\
	\nonumber
 	&:= H_1 + H_2.
\end{align}
Noting that $\E\|Z\|_s \leq \E\|H_1\|_s + \E\|H_2\|_s$, we will proceed by bounding each of $\E\|H_1\|_s$ and $\E\|H_2\|_2$ in turn.

%
%--------------------------------------------------------------
%
\item \label{step:H1s} {\bf Bound $\E\|H_1\|_s$.} 
We start by making the random variables in the expression for $H_1$ symmetric. Let $H_1'$ be an independent copy of $H_1$ created from an independent Gaussian sequence $\{g'_k(\omega)\}_{k,\omega}$, and set 
\begin{equation}
	\label{eq:Ydef}
	Y = H_1-H_1' =
	\sum_{k=1}^p\sum_{\omega=1}^m(|g'_k(\omega)|^2-|g_k(\omega)|^2)
	f_{k,\omega}f_{k,\omega}^*.
\end{equation}
Our strategy is to control $\|Y\|_s$ and use that fact $\E\|H_1\|_s\leq\E\|Y\|_s$, since
\begin{align*}
	\E\|H_1\|_s &= \E\|H_1 - \E H_1'\|_s &\text{($H_1'$ is zero mean)}\\
	&= \E\left\|\E[H_1-H_1'|H_1]\right\|_s 
	&\text{(independence, $\E[H_1'|H_1] = E[H_1']$)} \\
	&\leq \E\left[\E\left\|H_1-H_1'\right\|_s | H_1\right]
	&\text{(Jensen's inequality)} \\
	&= \E\|H_1-H_1'\|_s &\text{(iterated expectation).}
\end{align*}

Next, we randomize the sum in \eqref{eq:Ydef}. $Y$ has the same distribution as
\begin{equation}
	\label{eq:Ypdef}
	Y' = \sum_{k=1}^p\sum_{\omega=1}^m\epsilon_k(\omega)
	(|g'_k(\omega)|^2-|g_k(\omega)|^2 )f_{k,\omega}f_{k,\omega}^*,
\end{equation}
where $\{\epsilon_k(\omega)\}_{k,\omega}$ is an independent Rademacher sequence --- the $\epsilon_k(\omega)$ are iid and take values of $\pm 1$ with equal probability.  Note that
\[
	\E\|Y\|_s = \E\|Y'\|_s = 
	\E\left[\E\left[ \|Y'\|_s ~|~ \{g_k(\omega)\},\{g'_k(\omega)\} \right] \right].
\]

Third, apply Lemma~\ref{lm:rvmean} with $v_{k,\omega}=|~|g'_k(\omega)|^2-|g_k(\omega)|^2~|^{1/2}f_{k,\omega}$.  We define the random variable $B$ as 
\begin{equation}
	\label{eq:Bmaxmax}
	B := 	\max_{k,\omega}\max\left\{|g_k(\omega)|,|g'_k(\omega)|\right\} ~\geq~
	\max_{k,\omega} |~|g_k(\omega)|^2 - |g'_k(\omega)|^2~|^{1/2}, 
	%~\geq~\|v_{k,\omega}\|_\infty.
\end{equation}
and note that
\[
	\|v_{k,\omega}\|_\infty 
	~\leq~
	\max_{k,\omega} |~|g_k(\omega)|^2 - |g'_k(\omega)|^2~|^{1/2}\cdot\|f_{k,\omega}\|_\infty
	~\leq~
	B/\sqrt{m}.
\]
With the $\{g_k(\omega)\},\{g'_k(\omega)\}$ fixed, Lemma~\ref{lm:rvmean} (with $M=B/\sqrt{m}$) tells us that
\begin{align*}
	\E\left[ \|Y'\|_s ~|~ \{g_k(\omega)\},\{g'_k(\omega)\} \right] &\leq
	\sqrt{\frac{C\cdot s\cdot L(s,n,m,p)}{m}}\cdot B \cdot 
	\left\|\sum_{k=1}^p\sum_{\omega=1}^m (|g'_k(\omega)|^2-|g_k(\omega)|^2)
	f_{k,\omega}f_{k,\omega}^* \right\|_s^{1/2},
\end{align*}
where $L(s,n,m,p) = \log^2(s)\log(np)\log(mp)$ --- to make things more compact, we will abbreviate this with $L$, and remember that the quantity depends on the sparsity, length of the channel, length of the pulse, and number of channels.  Then by the Cauchy-Schwarz inequality
\begin{align}
	\label{eq:EYs}
	\E\|Y\|_s &\leq \sqrt{\frac{C\cdot s\cdot L}{m}}\cdot (E[B^2])^{1/2}\cdot
	\left(\E\left\|\sum_{k=1}^p\sum_{\omega=1}^m (|g'_k(\omega)|^2-|g_k(\omega)|^2)
	f_{k,\omega} f_{k,\omega}^* \right\|_s\right)^{1/2}. 
\end{align}

We can estimate $\E[B^2]$ as follows.  $B^2$ is the maximum of the $|g_k(\omega)|^2,|g'_k(\omega)|^2$, which are chi-squared random variables of degree 2 (when $2\leq\omega\leq m/2$) or 1 (when $\omega = 1,m/2+1$).  In either case, $\E[|g_k(\omega)|^2]=1$, and 
\[
	\P\left\{|g_k(\omega)|^2 ~>~ u \right\} ~\leq~ e^{-u},
\]
and since there are $(m/2+1)\cdot p\cdot 2 = (m+2)p$ unique magnitudes among the $|g_k(\omega)|^2,|g'_k(\omega)|^2$,
\begin{equation}
	\label{eq:PB2}
	\P\left\{B^2 ~>~ u\right\} ~\leq~ \min\left(1, (m+2)p\cdot e^{-u}\right).
\end{equation}
Since $B^2$ is a positive random variable
\begin{align}
	\label{eq:EB2}
	\E[B^2] &= \int_0^\infty \P\{B^2 > u\}~du \\
	\nonumber
	 &\leq \log ((m+2)p) + (m+2)p\int_{\log ((m+2)p)}^\infty e^{-u}~du \\
	\nonumber
	 &= \log ((m+2)p) + 1.
\end{align}
Combining this with the fact that
\[
	\E\left\|\sum_{k=1}^p\sum_{\omega=1}^m 
	\left(|g'_k(\omega)|^2-|g_k(\omega)|^2\right) f_{k,\omega} f_{k,\omega}^*\right\|_s 
	~\leq~
	2\E\left\|\sum_{k=1}^p\sum_{\omega=1}^m 
	|g_k(\omega)|^2 f_{k,\omega} f_{k,\omega}^*\right\|_s, 
\]
the bound in \eqref{eq:EYs} becomes
\begin{align*}
	\E\|Y\|_s &\leq  \sqrt{\frac{C\cdot s\cdot L\log(mp)}{m}}\cdot
	\left(\E\left\|\sum_k\sum_\omega |g_k(\omega)|^2f_{k,\omega}f_{k,\omega}^*
	\right\|_s\right)^{1/2}. 
\end{align*}
Using \eqref{eq:diagI} and the fact that $\E\|I\|_s = 1$ yields
\begin{align}
	\label{eq:EYp1}
	\E\|Y\|_s
	&\leq \sqrt{\frac{C\cdot s\cdot L\log(mp)}{m}}\cdot
	\left(\E\left\|\sum_k\sum_\omega (1-|g_k(\omega)|^2)
	f_{k,\omega}f_{k,\omega}^* \right\|_s + 1\right)^{1/2} \\ \nonumber
	&= \sqrt{\frac{C\cdot s\cdot L\log(mp)}{m}}\cdot
	\left(\E\|H_1\|_s + 1\right)^{1/2} \\ \nonumber
	&\leq \sqrt{\frac{C\cdot s\cdot L\log(mp)}{m}}\cdot
	\left(\E\|Y\|_s + 1\right)^{1/2}
\end{align}
Invoking Lemma~\ref{lm:simpleineq} with $\beta = \E\|Y\|_s$, $\alpha = \sqrt{CsL\log(mp)/m}$, and $c=0$, we see that there exist constants $C_{10},C_{11}$ such that when
\[
	m \geq C_{10}\cdot s\cdot L\log(mp)
\]
we will have
\begin{equation}
	\label{eq:E1bound}
	\E\|H_1\|_s \leq \E\|Y\|_s\leq C_{11}\cdot\sqrt{\frac{s\cdot L\log(mp)}{m}}.
\end{equation}

%--------------------------------------------------------------
%
\item \label{step:decoupling} {\bf Decouple $H_2$.}  Set
\begin{equation}
	\label{eq:H2p}
	H_2' = \sum_{j\not= k}\sum_{\omega=1}^m g_k(\omega)^*g'_j(\omega)
	f_{k,\omega}f_{j,\omega}^*,
\end{equation}
where $\{g'_k(\omega)\}$ is an independent sequence of random variables with the same distribution as $\{g_k(\omega)\}$.  We can now control $\|H_2\|_s$ by controlling $\|H_2'\|_s$, because 
\begin{equation}
	\label{eq:Edecouple}
	\E\|H_2\|_s ~\leq ~ C_{12} \E\|H_2'\|_s.
\end{equation}
For proof of \eqref{eq:Edecouple}, see \cite[Section 3.1]{pena99de}, which also provides a precise value for the constant $C_{12}$. 

%
%--------------------------------------------------------------
%
\item\label{step:Eadddiag} {\bf Add back the diagonal.}  Write
\begin{align}
	\label{eq:H3H4}
	H_2' &= \sum_{j=1}^p\sum_{k=1}^p\sum_{\omega=1}^m g_k(\omega)^*g'_j(\omega)
	f_{k,\omega}f_{j,\omega}^* - 
	\sum_{k=1}^p\sum_{\omega=1}^m g_k(\omega)^*g'_k(\omega)
	f_{k,\omega}f_{k,\omega}^* \\
	\nonumber
	&:= H_3 + H_4
\end{align}
%
%--------------------------------------------------------------
%
\item\label{step:EH4s} {\bf Bound $\E\|H_4\|_s$.}  Denoting the angle of the complex number $g_k(\omega)^*g'_k(\omega)$ as $\theta_k(\omega)$, $H_4$ has the same distribution as
\begin{equation}
	\label{eq:H4p}
	H'_4 = \sum_{k}\sum_\omega\epsilon_k(\omega) u_{k,\omega}v_{k,\omega}^*,
	\quad u_{k,\omega} = e^{\iunit\theta_k(\omega)/2}
	|g_k(\omega)| f_{k,\omega},\quad 
	v_{k,\omega} = e^{-\iunit\theta_k(\omega)/2}|g'_k(\omega)|f_{k,\omega}
\end{equation}
where $\{\epsilon_k(\omega)\}$ is an independent Rademacher sequence.  With $B$ as in \eqref{eq:Bmaxmax}, it follows that
\[
	\|u_{k,\omega}\|_\infty\leq B/\sqrt{m}
	\quad\text{and}\quad
	\|v_{k,\omega}\|_\infty\leq B/\sqrt{m}.
\]	
With the $\{g_k(\omega),g'_k(\omega)\}$ fixed, we apply Lemma~\ref{lm:rvnon} to get:
\begin{align*}
	\E[~\|H_4'\|_s~|~\{g_k(\omega)\},\{g'_k(\omega)\}] &\leq 
	\sqrt{\frac{C\cdot s\cdot L}{m}}\cdot B\cdot \\
	&\quad  \left(\left\|\sum_{k=1}^p\sum_{\omega=1}^m |g_k(\omega)|^2 
	f_{k,\omega}f_{k,\omega}^* \right\|_s^{1/2} + 
	\left\|\sum_{k=1}^p\sum_{\omega=1}^m |g'_k(\omega)|^2 
	f_{k,\omega}f_{k,\omega}^* \right\|_s^{1/2}\right).
\end{align*}
As in \eqref{eq:EYs}, we use the law of iterated expectation and the Cauchy-Schwarz inequality to remove the conditioning:
\begin{align*}
	\E\|H_4'\|_s &\leq 
	\sqrt{\frac{C\cdot s\cdot L}{m}}\cdot (\E[B^2])^{1/2}\cdot \\
	&\quad  \left(\E\left[\left(\left\|\sum_{k=1}^p\sum_{\omega=1}^m |g_k(\omega)|^2 
	f_{k,\omega}f_{k,\omega}^* \right\|_s^{1/2} + 
	\left\|\sum_{k=1}^p\sum_{\omega=1}^m |g'_k(\omega)|^2 
	f_{k,\omega}f_{k,\omega}^* \right\|_s^{1/2}\right)^2\right]\right)^{1/2}.
\end{align*}
The $\{g_k(\omega)\}$ and $\{g'_k(\omega)\}$ are identically distributed, and so using Jensen's inequality:
\[
	\E\|H_4'\|_s ~\leq~ \sqrt{\frac{C\cdot s\cdot L}{m}}\cdot (\E[B^2])^{1/2} \cdot
	\left(\E\left\|\sum_{k=1}^p\sum_{\omega=1}^m |g_k(\omega)|^2 
	f_{k,\omega}f_{k,\omega}^* \right\|_s\right)^{1/2}.
\]
Using the bound in \eqref{eq:EYp1} in Step~\ref{step:H1s}, we have
\begin{align*}
	\E\left(\left\|\sum_{k=1}^p\sum_{\omega=1}^m |g_k(\omega)|^2 
	f_{k,\omega}f_{k,\omega}^* \right\|_s^{1/2} \right) 
	%&\leq \left(\E\left\|\sum_{k=1}^p\sum_{\omega=1}^m |g_k(\omega)|^2 
	%f_{k,\omega}f_{k,\omega}^* \right\|_s\right)^{1/2} \\
	&\leq \left(\E\left\|\sum_{k=1}^p\sum_{\omega=1}^m (1-|g_k(\omega)|^2) 
	f_{k,\omega}f_{k,\omega}^* \right\|_s + 1 \right)^{1/2} \\
	&= \left(\E\|H_1\|_s + 1\right)^{1/2}.
%	&\leq \left(\sqrt{\frac{C\cdot s\cdot\log^2s\log^3mp}{m}} + 1\right)^{1/2} \\
%	&\leq \sqrt{2},
\end{align*}
Similar to \eqref{eq:E1bound} (but with different constants) we see that $m\geq C\cdot s\cdot L\log(mp)$ implies $\E\|H_1\|_s\leq 1$.  As we reasoned in \eqref{eq:EB2} above, we will also have $\E[B^2]\leq \log((m+2)p)+1$, and so there are constants $C_{13},C_{14}$ such that 
\begin{equation}
	\label{eq:EH4s}
	\E\|H_4\|_s \leq C_{14}\cdot \sqrt{\frac{s\cdot L\log(mp)}{m}}
\end{equation}
when
\[
	m \geq C_{13}\cdot s\cdot L\log(mp).
\]
%\begin{align*}
%	\E\|H_4\|_s = \E\|H_4'\|_s &\leq \sqrt{\frac{C\cdot s\cdot L}{m}}\cdot
%	(\E[B^2])^{1/2} \\
%	&\leq \sqrt{\frac{C\cdot s\cdot L\log(mp)}{m}}.
%\end{align*}

%
%--------------------------------------------------------------
%
\item \label{step:EH3s} {\bf Bound $\E\|H_3\|_s$.}  $H_3$ has the same distribution as
\begin{align*}
	H_3' &= \sum_{j=1}^p\sum_{k=1}^p\sum_{\omega=1}^m
	\epsilon_{\omega}g_k(\omega)^*g'_j(\omega)f_{k,\omega}f_{j,\omega}^* \\
	&= \sum_{\omega=1}^m \epsilon_\omega
	\left(\sum_{k=1}^p g_k(\omega)^*f_{k,\omega}\right)
	\left(\sum_{j=1}^p g'_j(\omega)f_{j,\omega}^*\right) \\
	&= \sum_{\omega=1}^m \epsilon_\omega u_\omega v_\omega^*,
\end{align*}
where 
\begin{equation}
	\label{eq:uv}
	u_\omega = \sum_{k=1}^p g_k(\omega)^*f_{k,\omega} 
	\quad\text{and}\quad
	v_\omega = \sum_{j=1}^p {g'_j(\omega)}^*f_{j,\omega}.  
\end{equation}
The $f_{k,\omega}$ have disjoint support for different values of $k$, so $\|u_\omega\|_\infty,\|v_\omega\|_\infty\leq B/\sqrt{m}$ where $B$ is defined as in \eqref{eq:Bmaxmax}.  Also note that
\[
	\sum_{\omega=1}^m u_\omega u_\omega^* = \sum_{\omega=1}^m\sum_{k=1}^p\sum_{j=1}^p 
	g_k(\omega)^*g_j(\omega)  f_{k,\omega}f_{j,\omega}^*,  
\]
and so recalling \eqref{eq:PhitPhi}, we see that $\sum_{\omega=1}^m u_\omega u_\omega^*$ and $\sum_{\omega=1}^m v_\omega v_\omega^*$ are independent realizations of $\Phi^*\Phi$.
Lemma~\ref{lm:rvnon} and Cauchy-Schwarz tell us that
\begin{align}
	\nonumber
	\E\|H_3\|_s = \E\|H_3'\|_s &\leq \sqrt{\frac{C\cdot s\cdot L'}{m}}\cdot
	(\E[B^2])^{1/2}\cdot (\E\|\Phi^*\Phi\|_s)^{1/2} \\
	\label{eq:EH3s}
 	&\leq C_{15}\cdot\sqrt{\frac{s\cdot L'\log(mp)}{m}}\cdot (\E\|Z\|_s + 1)^{1/2},
\end{align}
where $L':=L'(s,n,m,p):= \log^2s\log m\log np$.  Since $L'(s,n,m,p) < L(s,n,m,p)$, we can replace $L'$ with $L$ above.

%
%--------------------------------------------------------------
%
\item {\bf Collect the results.}  To summarize, we have shown that 
\begin{align*}
	\E\|Z\|_s &\leq \E\|H_1\|_s + \E\|H_2\|_s 
	&\quad\text{(from \eqref{eq:h1h2} in Step~\ref{step:Erank1})}\\
	&\leq \E\|H_1\|_s + C_{12}\E\|H_2'\|_s 
	&\quad\text{(from \eqref{eq:Edecouple} in Step~\ref{step:decoupling})} \\
	&\leq \E\|H_1\|_s + C_{12}C_{15}\E\|H_3\|_s + C_{12}C_{14}\E\|H_4\|_s 
	&\quad\text{(from \eqref{eq:H3H4} in Step~\ref{step:Eadddiag})}.
\end{align*}
For $m\geq \max(C_{10},C_{13})\cdot s\cdot L\log(mp)$, we also have the bounds
\begin{align*}
	\E\|H_1\|_s &\leq C_{11}\sqrt{\frac{sL\log(mp)}{m}} 
	&\quad\text{(from \eqref{eq:E1bound} in Step~\ref{step:H1s})} \\
	\E\|H_4\|_s &\leq C_{14}\sqrt{\frac{sL\log(mp)}{m}} 
	&\quad\text{(from \eqref{eq:EH4s} in Step~\ref{step:EH4s})}\\
	\E\|H_3\|_s &\leq C_{15}\sqrt{\frac{sL\log(mp)}{m}}\cdot(E\|Z\|_s+1)^{1/2}
	&\quad\text{(from \eqref{eq:EH3s} in Step~\ref{step:EH3s})}.
\end{align*}
Thus
\[
	\E\|Z\|_s \leq C\sqrt{\frac{sL\log(mp)}{m}}\cdot (2 + \sqrt{E\|Z\|_s+1}).
\]
Using Lemma~\ref{lm:simpleineq}, we see that there is indeed a constant $C_6$ such that when $m\geq C_6\cdot sL\log(mp)$, we will have
\begin{equation}
	\label{eq:EZfinal}
	\E\|I-\Phi^*\Phi\|_s = \E\|Z\|_s \leq \sqrt{\frac{C\cdot s\cdot L\log(mp)}{m}}.
\end{equation}

%
%--------------------------------------------------------------
%
\end{enumerate}

%% file: ProofTail.tex
We begin with a brief overview of the steps we will use to establish Theorem~\ref{th:Pmain}.  We will use the same decomposition of $Z$ as in Section~\ref{sec:proofmean}; dividing $Z$ into $Z=H_1+H_2$, decoupling $H_2$ to get $H_2'$, then dividing $H_2'$ into $H_2' = H_3+H_4$.  The essential idea is that we have estimated the means of $\|H_1\|_s$, $\|H_3\|_s$, and $\|H_4\|_s$ in the previous section; we will use these estimates and the concentration inequality in Lemma~\ref{lm:ltconc} to derive a tail bound for for each of these components in turn.

The main nuisance is that while we can write $H_1$, $H_3$, and $H_4$ as sums of independent random rank-1 matrices, the norms of these matrices are not bounded (as Gaussian random variables are not bounded).  To handle this, we define the random variable $B$ as in Section~\ref{sec:proofmean}
\[
	B = \max_{k,\omega}\max\{|g_k(\omega)|,|g'_k(\omega)|\},
\]
and then derive an estimate for $\|Z\|_s$ conditioned on the event
\[
	\mathcal{M} = \left\{B^2~\leq~ M\right\},
\]
where we will choose $M$ so that $\mathcal{M}$ likely to occur: $1/2\ll P\{\mathcal{M}\}<1$.  We will use $\E_\mathcal{M}$ and $\P_\mathcal{M}$ to denote expectation and probability conditioned on the event $\mathcal{M}$ occurring. 

We start by decomposing the tail bound as 
\[
	\P\{\|Z\|_s > \delta\} ~\leq~ 
	\P\{\|H_1\|_s > \delta/2\} + \P\{\|H_2\|_s > \delta/2\}.
\]
Step~\ref{step:PH1s} below bounds $\P\{\|H_1\|_s > \delta/2\}$.  Step~\ref{step:Pdecouple} decouples the sum for $H_2$ to get
\[
	\P\{\|H_2\|_s > \delta/2\} ~\leq~ C\P\{\|H_2'\|_s > \delta/2C\}.
\]
Steps~\ref{step:PH4s} and \ref{step:PH3s} then condition on $\mathcal{M}$, 
\[
	C\P\{\|H_2'\|_s > \delta/2C\} 
	~\leq~ 
	C\left(\P_\mathcal{M}\{\|H_2'\|_s > \delta/2C\} +\P\{\mathcal{M}^c\}\right),
\]
divide $H_2'$ into $H_2'=H_3+H_4$,
\[
	\P_\mathcal{M}\{\|H_2'\|_s > \delta/2C\}
	~\leq~ \P_\mathcal{M}\{\|H_3\|_s > \delta/4C\} + \P_\mathcal{M}\{\|H_4\|_s > \delta/4C \},
\]
and then bound $\P_\mathcal{M}\{\|H_3\|_s > \delta/4C\}$ and $\P_\mathcal{M}\{\|H_4\|_s > \delta/4C\}$ in turn.  These individual results are unified to finally establish the theorem in Step~\ref{step:Pcollect}.  

We will control each probability with a parameter $\gamma$, which can be selected as $0<\gamma<1/2$, and derive a bound for $m$ so that the total probability of failure is $O(\gamma)$.

\renewcommand{\theenumi}{{\bf P\arabic{enumi}}}
\begin{enumerate}
%
%--------------------------------------------------------------
%
\item {\bf Tail bound for $\|H_1\|_s$.}
\label{step:PH1s}
Recall the definitions of $Y=H_1-H_1'$ and $Y'$, which has the same distribution as $Y$, from \eqref{eq:Ydef} and \eqref{eq:Ypdef}.  We can develop a tail bound for $\|H_1\|_s$ from a tail bound for $\|Y\|_s$ (or $\|Y'\|_s$) by following \cite[Section 6.1]{ledoux91pr}.  For any $a,\lambda > 0$
\[
	\P\left\{\|H_1'\|_s < a\right\}\P\left\{\|H_1\|_s > a + \lambda \right\}
	~\leq~
	\P\left\{\|Y\|_s > \lambda\right\}.
\]
In particular, if we take $a = 2\E\|H_1\|_s=2\E\|H_1'\|_s$, we will have $\P\{\|H_1'\|_s \leq a\} \geq 1/2$, since the median of a positive random variable is no more than twice its mean, and so
\[
	\P\left\{\|H_1\|_s > 2\E\|H_1\|_s + \lambda \right\}
	~\leq~
	2\P\left\{\|Y\|_s > \lambda\right\}
	~=~
	2\P\left\{\|Y'\|_s > \lambda\right\}.
\]

To bound the right hand side, we first condition on $\mathcal{M}$:
\[
	\P\left\{\|Y'\|_s > \lambda\right\} 
	~\leq~
	\P_\mathcal{M}\left\{\|Y'\|_s > \lambda\right\} + \P\{\mathcal{M}^c\}.
\]
Conditioned on $\mathcal{M}$, each term in the sum that comprises $Y'$ has bounded norm, and so we can apply Lemma~\ref{lm:ltconc} with 
$V_{k,\omega} = |~|g'_k(\omega)|^2-g_k(\omega)|^2~|f_{k,\omega}f_{k,\omega}^*$,
noting that
\[
	\|V_{k,\omega}\|_s\leq M\cdot\|f_{k,\omega}f_{k,\omega}^*\|_s = M\cdot s/m.
\]
This yields
\[
	\P_\mathcal{M}\left\{\|Y'\|_s > 
	C\left(u\E_\mathcal{M}\|Y'\|_s + t\frac{Ms}{m}\right) \right\} 
	\leq e^{-u^2} + e^{-t}.
\]

From \eqref{eq:E1bound}, we know that
\[
	\E_\mathcal{M}\|Y'\|_s ~\leq~ \frac{\E\|Y'\|_s}{\P\{\mathcal{M}\}} ~\leq~ C\sqrt{\frac{sLM}{m}}
\]
when $m\geq C\cdot sL\log(mp)$.  Plugging this into the expression for the tail bound gives us
\[
	\P_\mathcal{M}\left\{\|Y'\|_s > 
	C\left(u\sqrt{\frac{sLM}{m}} + t\frac{sM}{m}\right) \right\} 
	\leq e^{-u^2} + e^{-t}.
\]
Take $u=\sqrt{\log 1/\gamma},~t=\log 1/\gamma$ to get
\begin{equation}
	\label{eq:Yptail}
	\P_\mathcal{M}\left\{\|Y'\|_s > \lambda \right\} 
	\leq 2\gamma,
	\qquad
	\lambda = C\left(\sqrt{\frac{sLM\log(1/\gamma)}{m}} + \log(1/\gamma) \frac{sM}{m}\right).
\end{equation}

With this value of $\lambda$, we can use the bound  \eqref{eq:E1bound} for $\E\|H_1\|_s$ to get
\[
	2\E\|H_1\|_s + \lambda 
	~\leq~ 
	C\left(\sqrt{\frac{sL\log(mp)}{m}} + 
	\sqrt{\frac{sLM\log(1/\gamma)}{m}} + \log(1/\gamma) \frac{sM}{m}\right). 
\]
Since we are choosing $m$ to make all three terms above less than $1$, the middle term will dominate.  We see that there is a constant $C_{16}$ such that
\[
	m~\geq~ C_{16}\cdot\delta_s^{-2}\cdot s\cdot LM\log(1/\gamma),
\]
implies
\[
	2\E\|H_1\|_s + \lambda 
	~\leq~ 
	\delta_{s}/2,
\]
and hence
\begin{equation}
	\label{eq:PH1s}
	\P\left\{\|H_1\|_s > \delta_s/2 \right\}
	~\leq~
	4\gamma + 2\P\{\mathcal{M}^c\}.
\end{equation}

%
%--------------------------------------------------------------
%
\item\label{step:Pdecouple} {\bf Decouple $H_2$.}
In Step~\ref{step:decoupling}, we saw that we could decouple $H_2$ and add back the diagonal, giving us the decomposition \eqref{eq:H3H4}.  We can also derive a tail bound for $\|H_2\|_s$ using the fact that
\[
	\P\{\|H_2\|_s \geq\lambda\} 
	~\leq~ 
	C_{17}\cdot\P\{\|H_2'\|_s\geq\lambda/C_{17}\}
\]
for a universal constant $C_{17}$, where $H_2'$ is the ``decoupled'' version of $H_2$ given by \eqref{eq:H2p} (for proof of this and an explicit value for $C_{17}$, see \cite[Section 3.4]{pena99de}).  We will decompose $H_2'=H_3+H_4$ as in \eqref{eq:H3H4} and proceed by finding tail bounds for $\|H_3\|_s$ and $\|H_4\|_s$ conditioned on $\mathcal{M}$.

%
%--------------------------------------------------------------
%
\item {\bf Conditional tail bound for $\|H_4\|_s$.}
\label{step:PH4s}
We start with the tail bound for $\|H_4\|_s$.  Recall that $H_4$ has the same distribution as $H_4'$ in \eqref{eq:H4p}.  Using \eqref{eq:EH4s} from Step~\ref{step:EH4s}, we can bound the conditional mean
\[
	\E_\mathcal{M}\|H'_4\|_s 
	~\leq~ 
	\frac{\E\|H'_4\|_s}{\P\{\mathcal{M}\}} 
	~\leq~
	C\sqrt{\frac{sL\log(mp)}{m}}.
\]	
Recall that we can write $H_4'$ as a random sum of rank-1 matrices as shown in \eqref{eq:H4p}.  Conditioned on $\mathcal{M}$, 
\[
	\|u_{k,\omega}v_{k,\omega}^*\|_s ~\leq~ 
	M \|f_{k,\omega}f_{k,\omega}^*\|_s = 
	M\cdot s/m. 
\]
We now apply the concentration inequality \eqref{eq:ltconc} as before with $u=\sqrt{\log(C_{17}/\gamma)}$ and $t=\log (C_{17}/\gamma)$:
\[
	\P_\mathcal{M}\left\{\|H_4'\|_s > C\left(\sqrt{\frac{s\cdot LM\log(C_{17}/\gamma)}{m}} + 
	\frac{s\cdot M\log(C_{17}/\gamma)}{m}\right)\right\} \leq 
	2\gamma/C_{17}.
\]
Since we are making both terms on the right-hand-side inside the probability brackets less than one, the first one will dominate.  Thus there exists a constant $C_{18}$ so that
\[
	m ~\geq~ C_{18}\cdot \delta_s^{-2}\cdot s\cdot LM\log(1/\gamma).
\]
implies
\[
	\P_\mathcal{M}\left\{\|H'_4\|_s > \delta_s/4C_{17} \right\} 
	~\leq~
	2\gamma/C_{17}, 
\]
and finally	
\begin{equation}
	\label{eq:PH4s}
	\P_\mathcal{M}\left\{\|H_4\|_s > \delta_s/4C_{17}\right\} ~\leq~2\gamma/C_{17},
\end{equation}
since $H_4$ and $H_4'$ have the same distribution.

%
%--------------------------------------------------------------
%
\item {\bf Conditional tail bound for $\|H_3\|_s$.}  
\label{step:PH3s}
As in Step~\ref{step:EH3s}, $H_3$ has the same distribution as
\[
	H'_3 = \sum_{\omega=1}^m\epsilon_\omega u_\omega v_\omega^*,
\]
with $u_\omega,v_\omega$ as in \eqref{eq:uv}.  In \eqref{eq:EH3s} in Step~\ref{step:EH3s}, we showed that
\[
	\E\|H'_3\|_s ~\leq~ C\sqrt{\frac{sL\log(mp)}{m}}\left(\E\|Z\|_s+1\right)^{1/2},
\]
and in \eqref{eq:EZfinal}, we showed that $\E\|Z\|_s<1$ for $m\geq CsL\log(mp)$.  So for this range of $m$, we have $\E\|H'_3\|_s\leq Cm^{-1/2}(sL\log(mp))^{1/2}$ and so
\[
	\E_\mathcal{M}\|H_3'\|_s ~\leq~ 
	\frac{\E\|H_3'\|_s}{\P\{\mathcal{M}\}}~\leq~ 
	C\sqrt{\frac{s\cdot L\log(mp)}{m}}.
\]
Conditioned on $\mathcal{M}$,
\[
	\|u_\omega v_\omega^*\|_s ~\leq~ 
	\sup_{\substack{|\SP|\leq s \\ y\in B^\SP_2}} |y^*u_\omega|\cdot
	\sup_{\substack{|\SP|\leq s \\ x\in B^\SP_2}} |x^*v_\omega|
	~\leq~
	M\cdot s/m.
\]
We apply the concentration inequality \eqref{eq:ltconc} with $u=\sqrt{\log(C_{17}/\gamma)}$ and $t=\log(C_{17}/\gamma)$, yielding
\[
	\P_\mathcal{M}\left\{\|H_3\|_s > C\left(\sqrt{\frac{s\cdot L\log(mp)\log(C_{17}/\gamma)}{m}} + 
	\frac{s\cdot M\log(C_{17}/\gamma)}{m}\right)\right\} 
	~\leq~ 2\gamma/C_{17}.
\]
Below we will see that we can take $M\sim \log(mp/\gamma)$; this means that there exists a constant $C_{19}$ such that
\[
	m ~\geq~ C_{19}\cdot \delta_s^{-2}\cdot s\cdot LM\log(1/\gamma)
\]
implies
\begin{equation}
	\label{eq:PH3s}
	\P_\mathcal{M}\left\{\|H_3\|_s > \delta_s/4C_{17}\right\} ~\leq~ 2\gamma/C_{17}.
\end{equation}

%
%--------------------------------------------------------------
%
\item\label{step:Pcollect} {\bf Collect the tail bounds.}  

We have shown that
\begin{align*}
	\P\left\{\|Z\|_s > \delta_s\right\} &\leq
	\P\left\{\|H_1\|_s > \delta_s/2\right\} + \P\left\{\|H_2\|_s > \delta_s/2\right\} \\
	&\leq \P\left\{\|H_1\|_s > \delta_s/2\right\} + C_{17}\P\left\{\|H_2'\|_s > \delta_s/2C_{17}\right\}\\
	&\leq \P\left\{\|H_1\|_s > \delta_s/2\right\} + 
	C_{17}\P_\mathcal{M}\left\{\|H_2'\|_s > \delta_s/2C_{17}\right\} + C_{17}\P\{\mathcal{M}^c\} \\
	&\leq \P\left\{\|H_1\|_s > \delta_s/2\right\} +
	C_{17}\P_\mathcal{M}\left\{\|H_3\|_s > \delta_s/4C_{17}\right\} + \\
	&\qquad C_{17}\P_\mathcal{M}\left\{\|H_4\|_s > \delta_s/4C_{17}\right\} +
	C_{17}\P\{\mathcal{M}^c\}
\end{align*}

% We have shown that
% \begin{align}
% 	\label{eq:Pcondtotal}
% 	\P_\mathcal{M}\left\{\|Z\|_s > \delta_s\right\} &\leq
% 	\P_\mathcal{M}\left\{\|H_1\|_s > \delta_s/2\right\} +
% 	\P_\mathcal{M}\left\{\|H_2\|_s > \delta_s/2\right\} \\
% 	\nonumber
% 	&\leq \P_\mathcal{M}\left\{\|H_1\|_s > \delta_s/2\right\} +
% 	C_{17}\P_\mathcal{M}\left\{\|H_3\|_s > \delta_s/4C_{17}\right\} \\
% 	&\qquad\qquad  + C_{17}\P_\mathcal{M}\left\{\|H_4\|_s > \delta_s/4C_{17}\right\}.
% \end{align}

Combining the results from Steps~\ref{step:PH1s},\ref{step:PH4s}, and \ref{step:PH3s}, we see that for any $0<\gamma<1/2$, there is a constant $C_{20}$ such that when 
\[
	m\geq C_{20} \cdot \delta_s^{-2}\cdot s \cdot LM\log(1/\gamma)
\]
we will have all of the following
\begin{align*}
	\P\left\{\|H_1\|_s > \delta_s/2\right\} &\leq 4\gamma + 2\P\{\mathcal{M}^c\},
	&\quad\text{(from \eqref{eq:PH1s} in Step~\ref{step:PH1s})} \\
	C_{17}\P_\mathcal{M}\left\{\|H_4\|_s > \delta_s/4C_{17}\right\} &\leq 2\gamma,
	&\quad\text{(from \eqref{eq:PH4s} in Step~\ref{step:PH4s})} \\
	C_{17}\P_\mathcal{M}\left\{\|H_3\|_s > \delta_s/4C_{17}\right\} &\leq 2\gamma,
	&\quad\text{(from \eqref{eq:PH3s} in Step~\ref{step:PH3s}).}
\end{align*}

It remains to fix $M$.  As $\P\{B^2 > M\}\leq \min(1,~(m+2)p\cdot e^{-M})$ --- recall \eqref{eq:PB2} --- choosing 
\[
	M = \log(C_{17}(m+2)p/\gamma) \quad\Rightarrow\quad C_{17}\P\{\mathcal{M}^c\}\leq \gamma.
\]
With this choice of $M$ (and assuming that $C_{17}\geq 2$),
\begin{equation}
	\label{eq:PZfinal}
	\P\left\{\|Z\|_s > \delta_s\right\} ~\leq~ 10\gamma
\end{equation}
when
\begin{equation}
	\label{eq:mPfinal}
	m \geq C\cdot\delta_s^{-2}\cdot s\cdot L\log(mp/\gamma)\log(1/\gamma). 
\end{equation}
We establish the theorem by taking $\gamma = C(np)^{-1}$ and noting that
\[
	L\log(mp/\gamma)\log(1/\gamma) ~\leq~ \log^6 (np)
\]
and so taking $m$ as in \eqref{eq:m-tail} will guarantee \eqref{eq:mPfinal} and hence \eqref{eq:PZfinal}.

\end{enumerate}

%% file: RandomMatrices.tex
\subsection{Random sums of rank-1 matrices}

The theoretical results in this paper depend crucially on our ability to estimate the size of the $\|\cdot\|_s$ norm of random matrices that can be written as the sum of independent rank-1 matrices:
\begin{equation}
	\label{eq:randomsum1}
	\left\|\sum_{i=1}^m \epsilon_i u_iv_i^*\right\|_s,
\end{equation}
where the $u_i$ and $v_i$ are vectors in $\C^n$ and the $\{\epsilon_i\}$ are iid Bernoulli random variables taking the values $\pm 1$ with equal probability.  Taking $U,V$ as the $n\times m$ matrices with the $v_i,u_i$ as columns, and letting $\Sigma$ be the diagonal matrix with $\Sigma_{ii} = \epsilon_i$, \eqref{eq:randomsum1} can be written more compactly as $\|U\Sigma V^*\|_s$.

In \cite{rudelson08sp}, Rudelson and Vershynin provided a bound for the expectation of \eqref{eq:randomsum1} when $U=V$.  The following is Lemma~3.8 in \cite{rudelson08sp}:
\begin{lemma}
\label{lm:rvmean}
Let the vectors $v_i$ and the matrices $V$ and $\Sigma$ be defined as above, and suppose that $\|v_i\|_\infty\leq M$. 
Then for some constant $C$, 
\begin{equation}
	\label{eq:rvmean}
	\E\left\|V\Sigma V^*\right\|_s~\leq~
	C\cdot M\cdot s^{1/2}\cdot\log s\sqrt{\log m\log n}\cdot 
	\left\|VV^*\right\|_s^{1/2}.
\end{equation}
\end{lemma}

The following is the analogous result for the more general case when $U\not= V$:
\begin{lemma}
\label{lm:rvnon}
Let $V,\Sigma$ be as in Lemma~\ref{lm:rvmean}, and let $U$ be another $n\times m$ matrix whose maximum entry is less than $M$.  Then for some constant $C$, 
\[
	\E\|U\Sigma V^*\|_s ~\leq~
	C\cdot M\cdot s^{1/2}\cdot\log s\sqrt{\log m\log n}\cdot
	\left(\|VV^*\|_s^{1/2} + \|UU^*\|_s^{1/2}\right).
\]
\end{lemma}
\begin{proof}
As in \cite{rudelson08sp}, we can bound $\|U\Sigma V^*\|_s$ by the supremum of a Gaussian random process.  Letting $\{g_i\}$ be a sequence of iid Gaussian random variables with zero mean and unit variance, we have 
\begin{align}
	\nonumber
	\E\|U\Sigma V^*\| &= \E\sup_{\substack{|\SP|\leq s\\x_a,x_b\in B^\SP_2}}
	\left|\sum_{i=1}^m\epsilon_i\<x_a,u_i\>\<v_i,x_b\> \right| \\
	\label{eq:gproc}
	&\leq C\E\sup_{\substack{|\SP|\leq s\\x_a,x_b\in B^\SP_2}}
	\left|\sum_{i=1}^m g_i\<x_a,u_i\>\<v_i,x_b\> \right|.
\end{align}
We now apply the Dudley inequality (see, for example, \cite[Chapter 2]{talagrand05ge}), which states that for a Gaussian process $G(x)$ indexed by a set $x\in T$, the expected maximum value of $G$ over $T$ obeys
\begin{equation}
	\label{eq:dudley}
	\E\sup_{x\in T} |G(x)| ~\leq~ C\int_0^\infty \log^{1/2} N(T,\delta,t) dt.
\end{equation}
Above, $N(T,\delta,t)$ is the $t$-covering number for $T$ under the metric $\delta(x,y) = (\E[|G(x)-G(y)|^2])^{1/2}$.  The process in \eqref{eq:gproc} is indexed by two vectors $x_a,x_b$, so here
\[
	G(x_a,x_b) = \sum_i g_i \<x_a,u_i\>\<v_i,x_b\>,
	\quad\text{and}\quad
	T = \bigcup_{|\SP|\leq s}B_2^\SP\otimes B_2^\SP,
\] 
with the metric $\delta$ given by
\[
	\delta\left((x_a,x_b),(y_a,y_b)\right) = \left(\sum_{i=1}^m 
	\left(\<x_a,u_i\>\<v_i,x_b\> - \<y_a,u_i\>\<v_i,y_b\>\right)^2\right)^{1/2}.
\]
We can bound this distance using
\begin{align*}
	\delta\left((x_a,x_b),(y_a,y_b)\right) &=
 	\frac{1}{2}\left(\sum_{i=1}^m \left(\<x_a+y_a,u_i\>\<v_i,x_b-y_b\> + 	
	\<x_a-y_a,u_i\>\<v_i,x_b+y_b\> \right)^2\right)^{1/2} \\
 	&\leq \frac{1}{2}\cdot\max_i\left(|\<u_i,x_a-y_a\>|,~|\<v_i,x_b-y_b\>|\right)
	\left(\sum_{i=1}^m\left(|\<x_a+y_a,u_i\>| + |\<x_b+y_b,v_i\>| \right)^2
	\right)^{1/2} \\
 	&\leq R\cdot \max_i\left(|\<u_i,x_a-y_a\>|,~|\<v_i,x_b-y_b\>|\right),
\end{align*}
where 
\begin{align*}
	R^2 &= \frac{1}{4}\sup_{(x_a,x_b)\in T} 
	\sum_{i=1}^m\left(|\<x_a+y_a,u_i\>| + |\<x_b+y_b,v_i\>| \right)^2 \\
	&\leq \frac{1}{4}\sup_{(x_a,x_b)\in T}\left(
	\sum_{i=1}^m |\<x_a+y_a,u_i\>|^2 + 
	\sum_{i=1}^m |\<x_a+y_a,v_i\>|^2 +
	2\sum_{i=1}^m|\<x_a+y_a,u_i\>| \cdot |\<x_b+y_b,v_i\>|\right) \\
	&\leq \|UU^*\|_s + \|VV^*\|_s + 
	\frac{1}{2}\sup_{(x_a,x_b)\in T} 
	\sum_{i=1}^m|\<x_a+y_a,u_i\>| \cdot |\<x_b+y_b,v_i\>| \\
	&\leq \|UU^*\|_s + \|VV^*\|_s + 
	\frac{1}{2}\sup_{(x_a,x_b)\in T}
	\left(\sum_{i=1}^m|\<x_a+y_a,u_i\>|^2\right)^{1/2}\cdot
	\left(\sum_{i=1}^m|\<x_b+y_b,v_i\>|^2\right)^{1/2} \\
	&\leq \|UU^*\|_s + \|VV^*\|_s + 2\|UU^*\|_s^{1/2}\|VV^*\|_s^{1/2} \\
	&= \left(\|UU^*\|_s^{1/2} + \|VV^*\|_s^{1/2} \right)^2 =: R'^2.
\end{align*}
Defining the norms
\[
	\|x\|_U = \max_i|\<x,u_i\>|\quad\text{and}\quad \|x\|_V = \max_i|\<x,v_i\>|,
\]
our bound on the metric becomes
\[
	\delta\left((x_a,x_b),(y_a,y_b)\right) \leq
	R'\max\left(\|x_a-y_a\|_U,\|x_b-y_b\|_V\right).
\]

Now let 
\begin{equation}
	\label{eq:Ts}
	T' = \bigcup_{|\SP|\leq s}B^\SP_2
\end{equation}
and note that $T\subset T'\otimes T'$, and so $N(T,\delta,t)\leq N(T'\otimes T', \delta, t)$.  If $\mathcal{C}_1$ is a $t$-cover for $T'$ under the metric $\|\cdot\|_U$ and $\mathcal{C}_2$ is a $t$-cover for $T'$ under the metric $\|\cdot\|_V$, then $\mathcal{C}_1\otimes \mathcal{C}_2$ is a $t$ cover for $T'\otimes T'$ under the metric $\max(\|\cdot\|_U,\|\cdot\|_V)$.  Hence
\[
	N(T,\delta,t) ~\leq~ N(T',R'\|\cdot\|_U,t)\cdot N(T',R'\|\cdot\|_V,t)
\]
and
\begin{equation}
	\label{eq:newintegrals}
	\int_0^\infty \log^{1/2} N(T,\delta,t)dt ~\leq~
	R'\int_0^\infty  \log^{1/2}N(T',\|\cdot\|_U,t)dt ~+~
	R'\int_0^\infty  \log^{1/2}N(T', \|\cdot\|_V,t)dt.
\end{equation}

We can now apply estimates for the covering numbers in \eqref{eq:newintegrals} that were developed in \cite{rudelson08sp}, where the following is shown.
\begin{proposition}
	Let $x_1,\ldots,x_m$ be vectors in $\C^n$ with $\|x_i\|_\infty\leq M$, and define the norm $\|x\|_X = \max_i |\<x,x_i\>|$.  For $T'$ as in \eqref{eq:Ts}, 
\[
	\int_0^\infty \log^{1/2} N(T',\|\cdot\|_X,t) dt \leq C\cdot M\sqrt{s}\cdot
	\log s\log^{1/2}m\log^{1/2}n
\]
for some constant $C$. 
\end{proposition}
Combining this proposition with \eqref{eq:newintegrals} yields
\[
	\E\|U\Sigma V^*\|_s \leq C\cdot R'\cdot M \cdot \sqrt{s}\cdot
	\log s\log^{1/2}m\log^{1/2}n,
\]
establishing the lemma.
\end{proof}

\subsection{A concentration inequality}

The following is a specialized version of \cite[Th.~6.17]{ledoux91pr}, and appears in the following form in \cite[Prop.~19]{tropp10be}.
\begin{lemma}
\label{lm:ltconc}
Let $V_1,\ldots,V_m$ be a sequence of square matrices with $\|V_i\|_s\leq M$, and let $\{\epsilon_i\}$ be a Rademacher seqeunce.  Set $V=\sum_i\epsilon_iV_i$.  Then for all $u,t\geq 1$
\begin{equation}
\label{eq:ltconc}
\P\left\{\|V\|_s ~\geq~ C(u\E\|V\|_s + tM)\right\} ~\leq~ e^{-u^2} + e^{-t}.
\end{equation}
\end{lemma}

%% file: SimpleInequality.tex
\begin{lemma}
	\label{lm:simpleineq}
	Fix $\alpha\leq 1$ and $c \geq 0$.  If
	\begin{equation}
		\label{eq:bgiven}
		\beta \leq \alpha\left(c + \sqrt{\beta+1}\right) \quad\text{for}~\beta \geq 0,
	\end{equation}
	then
	\[
		\beta \leq \alpha\left(c + 1/2 + \sqrt{c + 5/4}\right)
		\quad\text{for}~\beta \geq 0.
	\]
\end{lemma}	

\begin{proof}
	Let $x = (\beta+1)^{1/2}$; note that $x$ is a monotonic function of $\beta$.  Then \eqref{eq:bgiven} becomes
	\[
		x^2 - 1 \leq \alpha (c+ x)  \quad\Rightarrow\quad x^2 - \alpha x - (\alpha c + 1) \leq 0.
	\]
	The polynomial on the left is strictly increasing when $x\geq \alpha/2$.  Since $\alpha\leq 1$ and $x\geq 1$ for $\beta \geq0$, it is strictly increasing over the entire domain of interest.  Thus
	\[
		x^2 - \alpha x - (\alpha c + 1) \leq 0 
		\quad\Rightarrow\quad
		x \leq \frac{\alpha + \sqrt{\alpha^2 + 4(\alpha c + 1)}}{2}.
	\]
	Substituting $(\beta+1)^{1/2}$ back in for $x$, this means
	\[
		\beta + 1 \leq \frac{\alpha^2}{4} + 
		\frac{\alpha\sqrt{\alpha^2+4(\alpha c + 1)}}{2} + 
		\frac{\alpha^2 + 4(\alpha c +1)}{4},
	\]
	and so
	\begin{align*}
		\beta &\leq \alpha\left(c + \frac{\alpha}{2} + \frac{\sqrt{\alpha^2+4(\alpha c + 1)}}{2}\right) \\
		&\leq \alpha\left(c + 1/2 + \sqrt{c + 5/4}\right)
	\end{align*}
	when $\alpha < 1$.
\end{proof}